\documentclass[reqno]{amsart}

\title{Surfaces with concentric or parallel $K$-contours}

\author{Shoichi Fujimori}%{S.~Fujimori}
 \address[Shoichi Fujimori]{%
  Department of Mathematics, Hiroshima University,
  Higashihiroshima, Hiroshima 739-8526, Japan
}
\email{fujimori@hiroshima-u.ac.jp}

\author{Yu Kawakami}%{Y.~Kawakami}
 \address[Yu Kawakami]{%
   Faculty of Mathematics and Physics,
   Kanazawa University,
   Kanazawa, 920-1192, Japan
}
\email{y-kwkami@se.kanazawa-u.ac.jp}

\author{Masatoshi Kokubu$^*$}%{M.~Kokubu}
\address[Masatoshi Kokubu]{%
   Department of Mathematics,
   School of Engineering,
   Tokyo Denki University,
   5 Senju-Asahi-Cho, Adachi-Ku
   Tokyo, 120-8551,
   Japan 
}
\email{kokubu@cck.dendai.ac.jp}
\date{April 21, 2024}
\thanks{$^*$Corresponding author}

\usepackage{amsmath}
\usepackage{amssymb}
\usepackage{amsthm}
\usepackage{enumerate}
\usepackage{graphicx}

\newcommand\fff{\, \mathrm{I}}
\newcommand\sff{\, \mathrm{I\!I}}

\newcommand\pd{\partial}
\newcommand\R{{\mathbb R}}

\newcommand\xbf{{\boldsymbol x}}
\newcommand\nbf{\boldsymbol n}
\newcommand\pbf{\boldsymbol p}
\newcommand\vbf{\boldsymbol v}

\renewcommand\Re{\operatorname{Re}}
\renewcommand\Im{\operatorname{Im}}

\theoremstyle{plain}
\newtheorem{theorem}{Theorem}[section]

\newtheorem{corollary}[theorem]{Corollary}

\theoremstyle{definition}
\newtheorem{definition}[theorem]{Definition}
\newtheorem{remark}[theorem]{Remark}

\numberwithin{equation}{section}
\usepackage{version}	% required for `\comment' (yatex added)
\begin{document}

\begin{abstract}
Surfaces with concentric $K$-contours and  parallel $K$-contours
in Euclidean $3$-space are defined. 
Crucial examples are presented and characterization of them are given. 
\end{abstract}

\maketitle

%\footnote[0]{2000 {\it Mathematics Subject Classification.\/}
% 53A05; 53A10, 53A35, 53C45}

%%%%%%%%%%%%%%%%%%%%%%%%%%%%%%%%%%%%%%%
\section{Introduction}%%%%%%%%%%%%%%%%%
%%%%%%%%%%%%%%%%%%%%%%%%%%%%%%%%%%%%%%%
%By chance, the authors came to know that 
The contours of the Gaussian curvature
function $K$ on the graph surface 
\begin{equation}\label{eq:first-example}
 z= \frac{x}{x^2 + y^2} 
\end{equation}
in the Euclidean $3$-space $(\R^3; x,y,z)$ 
map to concentric circles on the $xy$-plane by orthogonal projection,  
so it would be permissible to say that the surface \eqref{eq:first-example}
has weak symmetry in some sense.  
We will refer to this property 
by saying a surface has \textit{concentric $K$-contours}. 
We can immediately note that helicoidal surfaces 
have the same property. (Here a \textit{helicoidal surface} 
is, by definition, a surface in $\R^3$ which is invariant under 
a one-parameter group of rigid screw motions;    
it is a generalization of both surfaces of revolution and right helicoids. 
A helicoidal surface is also called 
a \textit{generalized helicoid} (cf. \cite{docarmo})).   
We also found that the surface called 
a \textit{monkey saddle} has the same property.
(See Section 22.2 in \cite{gray},  
where the monkey saddle appears as an example for which 
the converse of Gauss' Theorema Egregium does not hold.)
In view of these circumstances, simple questions come to mind: 
\begin{enumerate}[(i)]
 \item 
 Are there any surfaces with concentric $K$-contours other than 
\eqref{eq:first-example}, helicoidal surfaces or the monkey saddle? 
 \item 
 Can we find all surfaces with concentric $K$-contours?   
\end{enumerate}
The authors searched the literature, but failed to find research on this.

One of our purposes is to provide a family of examples, 
denoted by $\xbf_{m,c}$ in this paper,  
which includes both \eqref{eq:first-example} and the monkey saddle. 
Another purpose is to give a partial answer to the question (ii). 
In fact, under a certain assumption, any surface with concentric $K$-contours 
must be a surface $\xbf_{m,c}$ or a helicoidal surface (Theorem \ref{thm:cKc-dAinv}).  

On the other hand, it has been an interesting problem to understand
how much the behavior of the Gauss map determines the surface.
For instance, Kenmotsu \cite{kenmotsu} showed a representation theorem for an arbitrary surface in $\R^3$ 
in terms of the Gauss map and the mean curvature function of the surface. 
In addition to this, Hoffman, Osserman and Schoen \cite{hos} proved that for a complete oriented surface of 
constant mean curvature in $\R^3$, if its Gauss image lies in some open hemisphere, 
then it is a plane; if the Gauss image 
lies in a closed hemisphere, then it is a plane or a right circular cylinder. In this paper,  
we will show that a behavior of the Gauss map, 
called \textit{semi-rotational equivariance},  
characterizes the surfaces $\xbf_{m,c}$
(Theorem \ref{thm:semi-rot-G}). 

\medskip

This paper also reports on the case where concentric circles are replaced by 
parallel straight lines. 
We say that a surface has \textit{parallel K-contours} if  
the contours of the Gaussian curvature function $K$ 
produce parallel straight lines on a plane by orthogonal projection. 

\medskip

We refer to standard textbooks \cite{docarmo}, \cite{kobayashi}, \cite{uy}, etc, 
for fundamental facts about surface theory. 

%%%%%%%%%%%%%%%%%%%%%%%%%%%%%%%%%%%%%%% 
%\section{Preliminaries}%%%%%%%%%%%%%%%%
%%%%%%%%%%%%%%%%%%%%%%%%%%%%%%%%%%%%%%%
\section{Surfaces with concentric $K$-contours}\label{sec:s-cKc}

Throughout this paper, we shall 
use the following notation and assumption:
$M$ denotes a connected, smooth $2$-manifold and 
$\xbf \colon M \to \R^3$ a smooth immersion.  
$K$ denotes the Gaussian curvature function on $M$.    
We set $M_k := \{ p \in M \mid K(p) =k \}$ for a real number $k$,  
and consider the family $\mathcal C := \{ M_k \}_{ k \in \R }$. 
It is always assumed that 
\textit{$M$ has no open subset where $\operatorname{grad} K =0$}
because 
we wish to study the case where $\mathcal C$ is formed by a family of curves.
%In other words, $M_k$ has no interior point for each fixed $k$. 
% the case where $M$ has locally constant Gaussian curvature is excluded.

\begin{definition}\label{def:cKc}
We say that 
$\xbf \colon M \to \R^3$  has \textit{concentric K-contours} if 
there exists a plane in $\R^3$ such that the orthogonal projection 
 $\pi \colon \R^3 \to P$ maps $\mathcal C$ to a family 
of concentric circles on $P$. 
\end{definition}

It is obvious that helicoidal surfaces have concentric $K$-contours. 

\subsection{ A non-helicoidal example}

Let $m$ be an integer not equal to $0, 1$, and
let $c$ be a non-zero real number. 
Consider a graph surface 
\begin{equation}\label{eq:cKlc-2}
\xbf_{m,c}(z) = \left( \Re z, \Im z, c \Re (z^m) \right) = 
\left( x,y, \frac{c}{2} \left\{ (x+iy)^m + (x-iy)^m \right\} \right)
\end{equation}
for $z=x+iy$.
Note that $\xbf_{-1,1}$ and $\xbf_{3,1}$ coincide with 
the surface \eqref{eq:first-example} and the monkey saddle, respectively. 
In terms of the polar coordinates $z=re^{i \theta}$, $\xbf_{m,c}$ is expressed as
\begin{equation}\label{eq:cKlc}
 \xbf_{m,c}(r,\theta) = 
\left( r \cos \theta, r \sin \theta, c r^m \cos m \theta \right).  
\end{equation}
See Figures~\ref{fg:circlep} and \ref{fg:circlen}.
%, and \ref{fg:circler}.
%
\begin{figure}[htbp] %%%%%%%%%%
\centering
 \includegraphics[width=.40\linewidth]{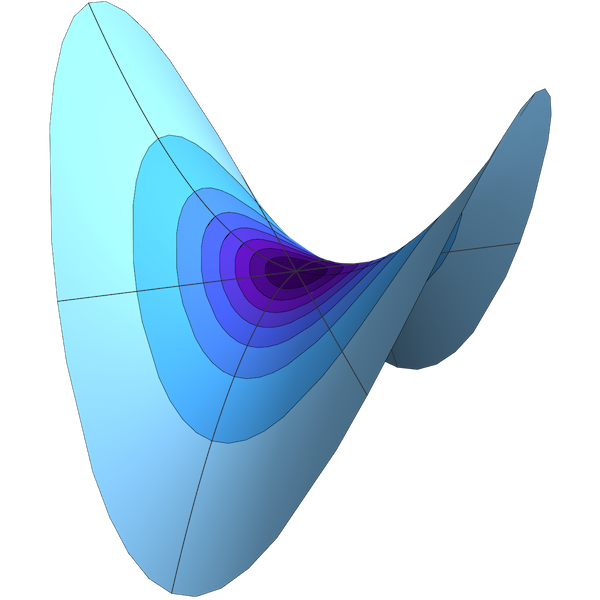} \ 
 \includegraphics[width=.40\linewidth]{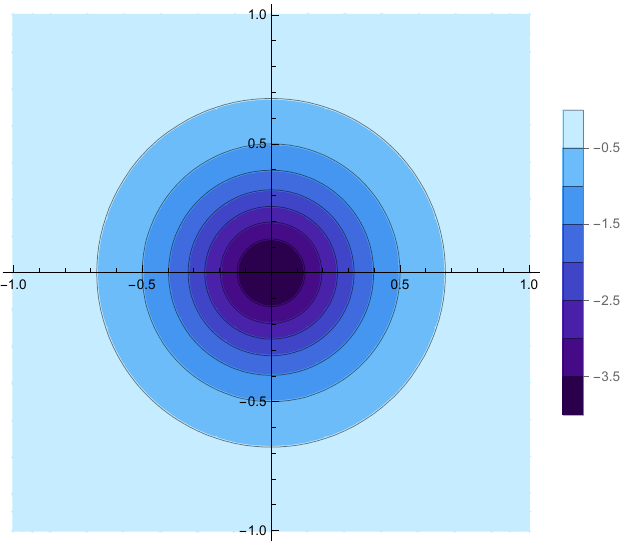} \\ 
 $c=1$, $m=2$ \\
 \includegraphics[width=.40\linewidth]{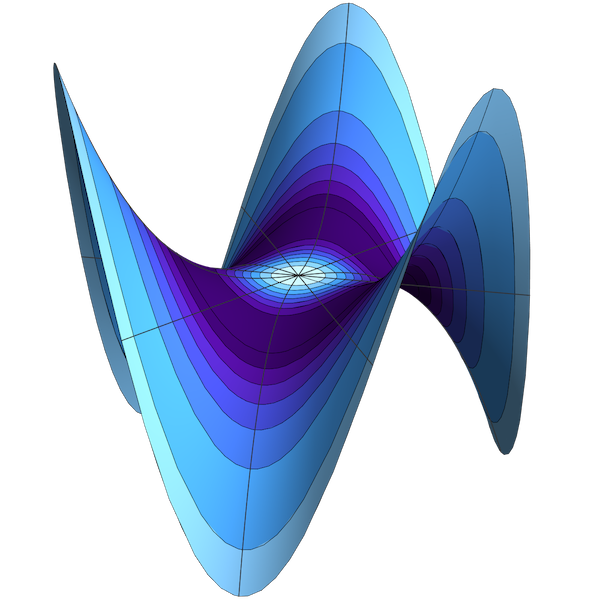} \ 
 \includegraphics[width=.40\linewidth]{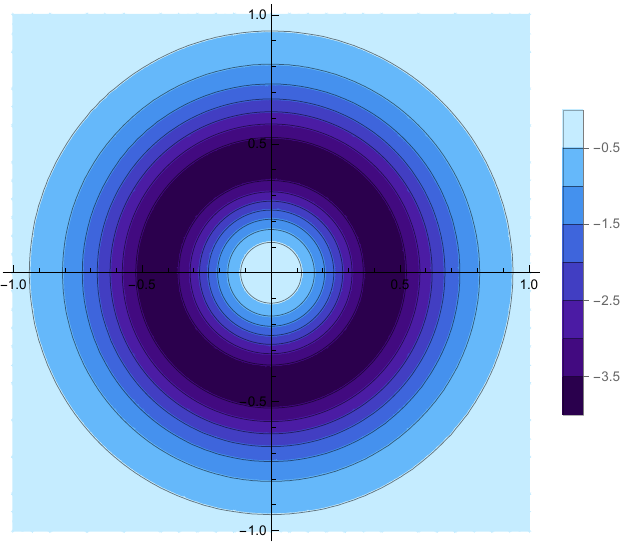} \\
 $c=1$, $m=3$ 
\caption{
The surfaces $\xbf_{m,c}$ (left) and their projections to the $xy$-plane (right) 
with positive integers $m$. 
%The surface $\xbf_{m,c}$ with $m=3$ coincides with the monkey saddle in \cite{gray}.
Each surface is gray-scaled by its Gaussian curvature.}
\label{fg:circlep} %%%%%%%%%%
\end{figure} 
\begin{figure}[htbp] %%%%%%%%%%
\centering
 \includegraphics[width=.40\linewidth]{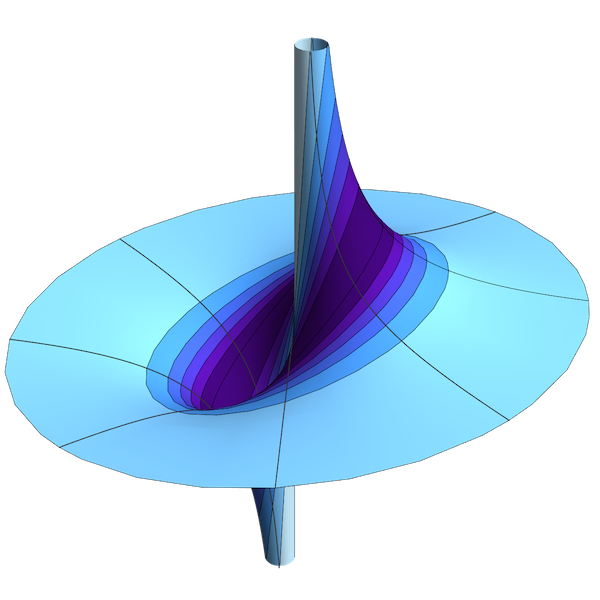} \ 
 \includegraphics[width=.40\linewidth]{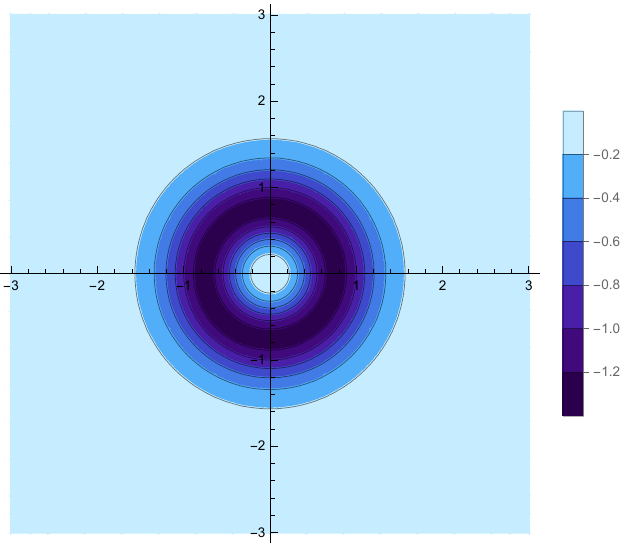} \\ 
 $c=1$, $m=-1$ \\
 \includegraphics[width=.40\linewidth]{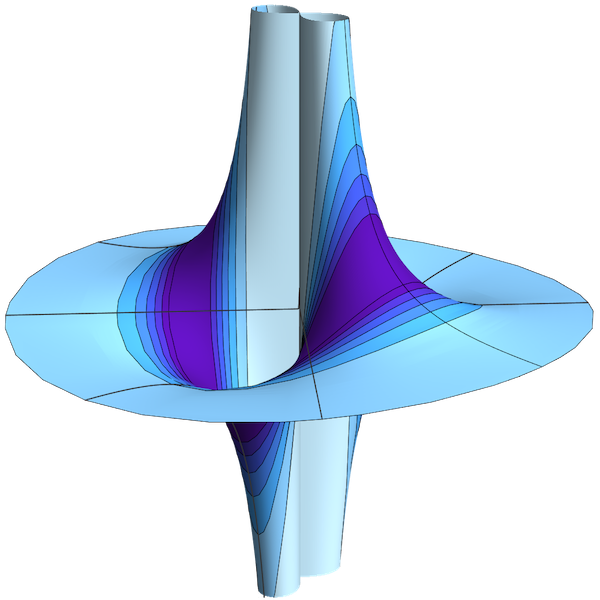} \ 
 \includegraphics[width=.40\linewidth]{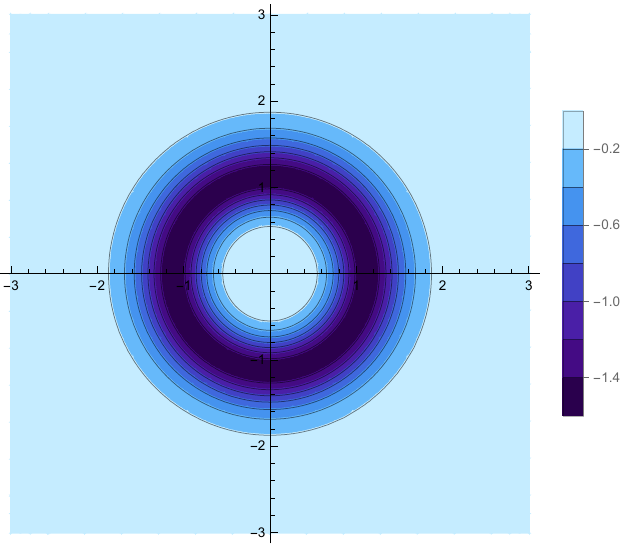} \\ 
 $c=1$, $m=-2$ 
\caption{
The surfaces $\xbf_{m,c}$ (left) and their projections to the $xy$-plane (right) 
with negative integers $m$. 
Each surface is gray-scaled by its Gaussian curvature.}
\label{fg:circlen} %%%%%%%%%%
\end{figure} 
\begin{figure}[htbp] %%%%%%%%%%
\centering
 \includegraphics[width=.40\linewidth]{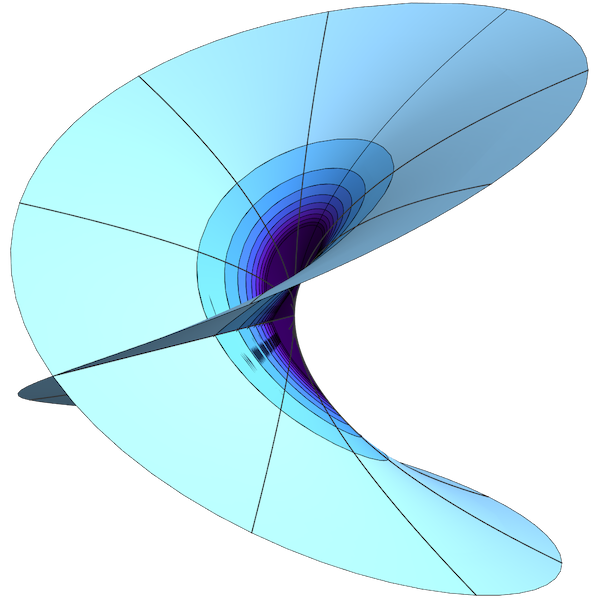} \ 
 \includegraphics[width=.40\linewidth]{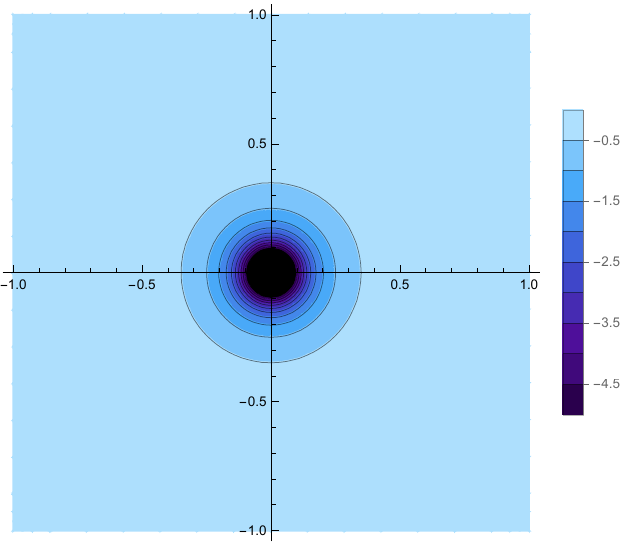} \\ 
 $c=1$, $m=1/2$ \\
 \includegraphics[width=.40\linewidth]{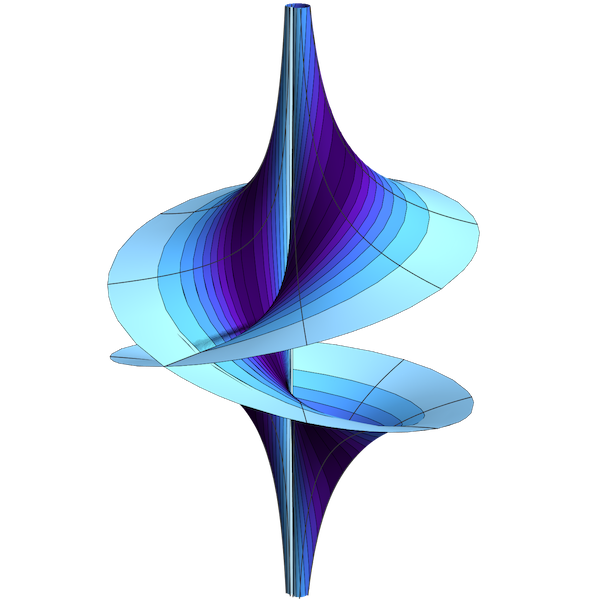} \ 
 \includegraphics[width=.40\linewidth]{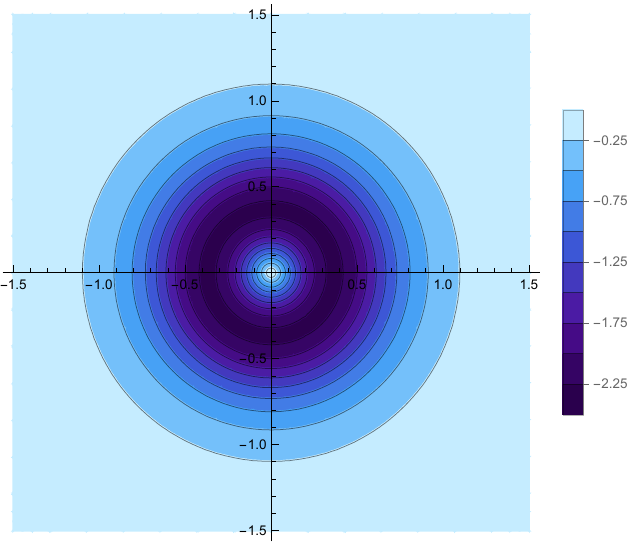} \\ 
 $c=1$, $m=-1/2$ 
\caption{
The surfaces $\xbf_{m,c}$ (left) and their projections to the $xy$-plane (right) 
with non integers $m$. 
Each surface is gray-scaled by its Gaussian curvature.}
\label{fg:circler} %%%%%%%%%%
\end{figure} 
The first and second fundamental forms $\fff, \sff$ and a unit normal $\nbf$ 
are as follows:  
\begin{align}
& \begin{aligned}
 \fff = & \, (1+c^2 m^2 r^{2m-2} \cos^2 m \theta)dr^2 + 
2 (-c^2 m^2 r^{2m-1} \cos m \theta \sin m \theta) dr d \theta \\ 
 & + (r^2 + c^2 m^2 r^{2m} \sin^2 m \theta) d \theta^2 ,  
\end{aligned} \notag
\\
& \nbf(r,\theta) =  \frac{1}{r \sqrt{1+ c^2 m^2 r^{2m-2}}}
\left( -c m r^m \cos (m-1)\theta, c m r^m \sin (m-1)\theta, 
r  \right) , \label{al:Gmap-xm} \\
& \sff = \frac{cm(m-1)}{r \sqrt{1+c^2 m^2 r^{2m-2}}}
\left\{ 
 r^{m-1} \cos m \theta dr^2 
-2r^{m} \sin m \theta dr d \theta 
 -r^{m+1} \cos m \theta d \theta^2
\right\}. \notag
\end{align}
From these, the Gaussian curvature $K$ and the mean curvature $H$ are 
\begin{align}
& K = K(r) = - \frac{c^2 m^2(m-1)^2 r^{2m-4}}{(1+c^2 m^2 r^{2m-2})^2}, \label{al:Gauss-cKc} \\ 
& H = H(r, \theta) = 
- \frac{c^3 m^3(m-1) r^{3m-4} \cos m \theta}{2 (1+c^2 m^2 r^{2m-2})^{3/2}}. 
\notag 
\end{align}
It follows directly from \eqref{al:Gauss-cKc} that 
$\xbf_{m,c}$ has concentric $K$-contours with respect to the $xy$-plane.    
Note that the first fundamental form $\fff$ does not have  
rotational symmetry but the Gaussian curvature $K$ does. 

\begin{remark}
\begin{enumerate}
 \item It follows from \eqref{eq:cKlc-2} that $\xbf_{m,c}$ is an entire graph over 
the $xy$-plane if $m$ is a positive integer.
 In particular, 
$\xbf_{m,c}$ is a hyperbolic paraboloid if $m=2$ and 
a monkey saddle if $m=3$. 
In the case where $m$ is a negative integer, 
$\xbf_{m,c}$ is a graph punctured at the origin. 
\item 
Although $\xbf_{m,c}$ can be defined for $m=0,1$ or $c=0$, it is a plane hence 
has constant Gaussian curvature zero. 
Therefore we exclude the case $m=0,1$ and the case $c=0$.
\end{enumerate} 
\end{remark}

It follows from \eqref{eq:cKlc} that $\xbf_{m,c}$ can be defined even if 
$m$ is a non-integer as a multi-valued graph over $\R^2 \setminus \{ (0,0) \}$
or a surface defined on the universal cover. See Figure 3.
From now on, we assume that \textit{the number $m$ for $\xbf_{m,c}$ 
does not have to be an integer, that is, 
%$m$ is a real number not equal to $0,1$.
$m \in \R \setminus \{ 0,1 \}$. 
}  

\subsection{Semi-rotational equivariance}
We also call the unit normal \eqref{al:Gmap-xm}
the \textit{Gauss map} of $\xbf_{m,c}$ according to custom. 
One can see from \eqref{al:Gmap-xm} that 
\begin{equation*}
 \nbf (r, \theta + \alpha) = \mathcal R_{(1-m)\alpha} \circ \nbf (r , \theta), 
\end{equation*}
where $\mathcal R_{(1-m)\alpha}$ denotes the rotation of angle $(1-m)\alpha$ 
with respect to the $z$-axis. 
Focusing on this property, we give the following definition:
\begin{definition}
 A surface $\xbf \colon M \to \R^3$ is said to have 
\textit{semi-rotational Gauss map} if there exist 
a straight line %a line 
$l \subset \R^3$, 
a plane $P \subset \R^3$, and a $1$-parameter group 
$\{ \phi_t \}$ of diffeomorphisms of $M$ such that 
\begin{enumerate}
 \item $l$ is orthogonal to $P$, 
 \item $\pi \circ \xbf \circ \phi_t = R_t \circ \pi \circ \xbf$, and
 \item $\nbf \circ \phi_t = \mathcal R_{k t} \circ \nbf$ 
for some constant $k$
\end{enumerate}
with a suitable choice of orientations of $l$ and $P$,  
where $\pi \colon \R^3 \to P$ is orthogonal projection,  
$R_t$ denotes a rotation on $P$ of angle $t$ with the center $P \cap l$, 
and $\mathcal R_{k t}$ denotes a rotation in $\R^3$ of angle $kt$  
with respect to the axis $l$. 
\end{definition}
 
Note that a helicoidal surface has semi-rotational Gauss map with $k=1$, 
which should be said to have \textit{rotational} Gauss map. 
So we shall use the term \textit{`strictly semi-rotational'} 
in the sense of `semi-rotational but not rotational'. 

\subsection{Characterizations of the surface $\xbf_{m,c}$}
\begin{theorem}\label{thm:cKc-dAinv}
 Let $\xbf \colon M \to \R^3$ be a surface with concentric $K$-contours. 
If the area element $dA$ is invariant along each $K$-contour, 
then $\xbf$ is a helicoidal surface or 
locally congruent to a surface $\xbf_{m,c}$
for some $m,c$. 
\end{theorem}

\begin{theorem}\label{thm:semi-rot-G}
 Let a surface $\xbf \colon M \to \R^3$ have semi-rotational Gauss map.  
Then $\xbf$ is a helicoidal surface or 
locally congruent to a surface $\xbf_{m,c}$ for some $m,c$. 
\end{theorem}
\begin{corollary}\label{cor:semi-rot-G}
 Let a surface $\xbf \colon M \to \R^3$ have strictly 
semi-rotational Gauss map.  
Then $\xbf$ is 
locally congruent to a surface $\xbf_{m,c}$ for some $m,c$. 
\end{corollary}

Before proving the theorems above, we write down formulas 
for the area element $dA$, the Gaussian curvature $K$ and 
the unit normal field $\nbf$ for a surface  
$\xbf(r, \theta) = \left( r \cos \theta, r \sin \theta, F(r,\theta) \right)$ : 
\begin{align}
 dA &= \varDelta \, dr \wedge d \theta, \label{al:dA=Delta} \\ 
 \nbf &= \frac{1}{\varDelta} \left( F_\theta \sin \theta -r F_r \cos \theta , \ 
- r F_r \sin \theta - F_\theta \cos \theta, \ r  \right),  
\label{al:unit-normal} \\ 
 K &= \frac{1}{\varDelta^4} \left\{ 
r^2F_{rr}( rF_r + F_{\theta \theta}) - (F_\theta - r F_{r \theta})^2 \right\} , 
\label{al:Gauss-curv}
\end{align}
where 
\begin{equation}\label{eq:Delta}
 \varDelta = \sqrt{ r^2 + r^2 F_r^2 + F_\theta^2  }. 
\end{equation}

\begin{proof}[Proof of Theorem \ref{thm:cKc-dAinv}]
Considering a rigid motion in $\R^3$, 
we may assume that the plane $P$ is the $xy$-plane and $K$-contours 
draw concentric circles with the center $(0,0)$ in the $xy$-plane. 
$\xbf$ is at least locally re-parameterized 
as $\xbf(r, \theta) = \left( r \cos \theta, r \sin \theta, F(r,\theta) \right)$. 
The function $\varDelta$ is of one variable $r$ because of  
\eqref{al:dA=Delta} and the assumption of invariance of $dA$. 
It follows from \eqref{eq:Delta} that $r^2 F_r^2 + F_\theta^2$ is also a function  
of one variable $r$. Therefore, there exist functions $\alpha = \alpha(r)$, 
$\beta = \beta(r, \theta)$ such that 
\begin{equation}\label{eq:Fth=ac-rFr=as}
rF_r = \alpha \cos \beta, \ F_\theta = \alpha \sin \beta . 
\end{equation}
By differentiating \eqref{eq:Fth=ac-rFr=as}, we have  
\begin{align} 
r^2 F_{rr} &=(r \alpha' - \alpha) \cos \beta - r \alpha \beta_r \sin \beta,  
\label{al:Frr} \\
 r (F_r)_\theta &= -\alpha \sin \beta \cdot \beta_\theta, \label{al:Frth} \\ 
(F_\theta)_r &= \alpha' \sin \beta + \alpha \cos \beta \cdot \beta_r, \label{al:Fthr} \\
F_{\theta \theta} &= \alpha \cos \beta \cdot \beta_\theta. \label{al:Fthth} 
\end{align}
It follows from \eqref{al:Frth}, \eqref{al:Fthr} that the equality 
$(F_r)_\theta = (F_\theta)_r$ turns out to be 
\begin{equation}\label{eq:bthr+br}
 \frac{\beta_\theta}{r} + \beta_r \frac{\cos \beta}{\sin \beta} = 
-\frac{\alpha'}{\alpha} . 
\end{equation}
Note that the right side of \eqref{eq:bthr+br} is of one variable $r$, 
so the left side is as well. Thus 
\begin{equation}\label{eq:pd-bthr+br}
 \frac{\pd}{\pd \theta} \left(
 \frac{\beta_\theta}{r} + \beta_r \frac{\cos \beta}{\sin \beta} 
\right) =0.
\end{equation}
On the other hand, using \eqref{al:Frr}--\eqref{al:Fthth}, we can rewrite  
\eqref{al:Gauss-curv} as 
\begin{equation*}
 K = \frac{\alpha(1 + \beta_\theta)(r \alpha' - \alpha)}{\varDelta^4}. 
\end{equation*} 
Here, $K$ must be a non-constant function of one variable $r$ 
by the assumption of concentric $K$-contours. 
It implies that $\beta_\theta$ is a function of one variable $r$.
% because $K=K(r)$ and $\varDelta = \varDelta (r)$. 
Therefore, we may set $\beta_\theta = \phi(r)$ and hence 
\begin{equation}\label{eq:beta=phithpsi}
 \beta = \phi(r) \cdot \theta + \psi(r)
\end{equation}
for some functions $\phi(r)$, $\psi(r)$.   
It follows from \eqref{eq:pd-bthr+br} with \eqref{eq:beta=phithpsi} that 
\begin{equation*}%\label{al:phi'(r)}
  \phi'(r) \{ \frac{1}{2} \sin 2\beta - \phi(r) \cdot \theta \} + \psi'(r) \phi(r)=0.
\end{equation*}
This implies that (i) $\frac{1}{2} \sin 2 \beta - \phi(r) \cdot \theta$ is 
independent of $\theta$ or (ii) $\phi'(r) =0$. In the case (i), 
by differentiating $\frac{1}{2} \sin 2 \beta - \phi(r) \cdot \theta$ by $\theta$, 
we have $(\cos 2 \beta -1)\phi(r) =0$, that is, 
\begin{equation}\label{eq:case(i)}
 \text{$\beta = n \pi$ for some integer $n$
or $\beta = \psi(r)$.}
\end{equation}
In the case (ii), the function $\phi$ is constant and $\psi' \phi =0$. 
Therefore $(\phi, \psi)=(k,l)$ for some constants $k,l$ or 
 $(\phi, \psi)=(0,\psi(r))$; in other words, 
\begin{equation}\label{eq:case(ii)}
 \text{$\beta = k \theta +l$ or $\beta = \psi(r)$.}
\end{equation}
Since the condition \eqref{eq:case(ii)} includes the condition \eqref{eq:case(i)}, 
we continue to discuss under the condition \eqref{eq:case(ii)}.  

\underline{In the case where $\beta = k \theta +l$}, the equation 
\eqref{eq:bthr+br} reduces to $\alpha' = -k \alpha / r$. 
Hence we have $\alpha = C r^{-k}$ for some constant $C$. 
It follows from \eqref{eq:Fth=ac-rFr=as} that $F= C_1 r^k \cos ( k \theta + l ) + C_2$ 
for some constants $C_1, C_2$. 
Thus the surface $\xbf$ is congruent to $\xbf_{k,C_1}$. 

\underline{In the case where $\beta = \psi(r)$}, the equation 
\eqref{eq:bthr+br} reduces to $\psi' \cot \psi = \alpha'/ \alpha $. 
It is solved as $\alpha \sin \psi = C$ for some constant $C$. 
The system of equations \eqref{eq:Fth=ac-rFr=as} turns out to be 
\begin{equation*}
 r F_r = \alpha(r) \cos ( \psi(r) ), \quad 
  F_\theta = C.  
\end{equation*} 
Therefore, we obtain  
\begin{equation*}
 F = C \theta + \int \frac{\alpha(r)}{r} \cos ( \psi(r) ) dr
 = C \theta + A(r)
\end{equation*}
for some function $A(r)$. Thus the surface $\xbf$ is helicoidal. 
\end{proof}

\begin{proof}[Proof of Theorem \ref{thm:semi-rot-G}]
Considering a rigid motion in $\R^3$, 
we may assume that the plane $P$ is the $xy$-plane and the line 
$l$ is the $z$-axis. 
The surface $\xbf$ is at least locally re-parameterized 
as $\xbf(r, \theta) = \left( r \cos \theta, r \sin \theta, F(r,\theta) \right)$. 
The Gauss map \eqref{al:unit-normal} is 
\begin{align*}
 \nbf = \frac{1}{\varDelta}
\begin{pmatrix}
F_\theta \sin \theta -r F_r \cos \theta \\  
- r F_r \sin \theta - F_\theta \cos \theta \\ 
r  
\end{pmatrix}
=
\begin{pmatrix}
 \cos \theta & - \sin \theta & 0 \\
 \sin \theta & \cos \theta & 0 \\
 0 & 0 & 1 
\end{pmatrix}
\frac{1}{\varDelta}
\begin{pmatrix}
 -r F_r \\ - F_\theta \\ r 
\end{pmatrix}
\end{align*}
in the column-vector form. Since $\nbf$ is semi-rotational, 
\begin{enumerate}[(i)]
 \item the vector-valued function
\begin{equation}\label{eq:v-v-ft}
\dfrac{1}{\varDelta}
\begin{pmatrix}
 -r F_r , - F_\theta ,  r 
\end{pmatrix} 
\end{equation}
 is of one variable $r$, or 
 \item 
there exist $m \in \mathbb R$ 
and $\phi_1=\phi_1(r), \psi = \psi(r)$  
such that 
\begin{equation*}
\frac{1}{\varDelta}
\begin{pmatrix}
 -r F_r , - F_\theta , r 
\end{pmatrix}
= 
\begin{pmatrix}
 \phi_1(r) \cos m \theta ,  \phi_1(r) \sin m \theta , \psi(r) 
\end{pmatrix}. 
\end{equation*}
\end{enumerate}

In the case (i), each component of \eqref{eq:v-v-ft} is of one variable $r$. 
Hence, $\varDelta$, $F_r$ and $F_\theta$ are functions of one variable $r$. 
This implies that $F$ must be of the form 
$F = a  \theta + \psi(r)$ for a constant $a$ and a function $\psi(r)$. 
Therefore 
$\xbf(r, \theta) = \left( r \cos \theta, r \sin \theta, a \theta + \psi(r) \right)$, 
that is, $\xbf$ is a helicoidal surface. 

\medskip

In the case (ii), the third component of \eqref{eq:v-v-ft} is of one variable $r$. 
Hence, $\varDelta$ is a function of one variable $r$.
Setting $-\phi_1(r) \cdot \varDelta = \varphi(r)$, we have 
\begin{equation}\label{eq:Fr-Fth}
\begin{cases}
 F_r = \frac{\varphi(r)}{r} \cos m \theta \\
 F_\theta = \varphi(r) \sin m \theta. 
\end{cases} 
\end{equation}
Thus the equality $(F_r)_\theta = (F_\theta)_r$ turns out to be 
\begin{equation}\label{eq:case-(ii)}
-m \frac{\varphi(r)}{r} \sin m \theta = \varphi'(r) \sin m \theta . 
\end{equation}

\underline{In the case where $m = 0$}, 
the system of equations \eqref{eq:Fr-Fth} 
turns out to be $F_r = \varphi(r)/r, F_\theta = 0$. Therefore, $F =F(r)$. 
This implies that $\xbf$ is a rotational surface.   

\underline{In the case where $m \ne 0$}, the equation \eqref{eq:case-(ii)} 
leads to $-m \frac{\varphi(r)}{r} = \varphi'(r)$. 
Therefore, $\varphi(r) = C r^{-m}$ ($C$ is a constant.) 
Then the solution to the system of equations \eqref{eq:Fr-Fth} is     
\begin{equation*}
 F(r, \theta) = C_1 r^{-m} \cos m \theta + C_2 \quad \text{($C_1, C_2$ are constants).}
\end{equation*}
Thus $\xbf$ is congruent to $\xbf_{-m, C_1}$.  
\end{proof}

\section{Surfaces with parallel $K$-contours}

We shall discuss 
here using the same notations and assumption
as in Section~\ref{sec:s-cKc}.

%Let $M$ be a $2$-manifold and $\xbf \colon M \to \R^3$ an immersion.  
%Let $K$ denote the Gaussian curvature function on $M$.    
%Set $M_k := \{ p \in M \mid K(p) =k \}$ for a real number $k$,  
%and consider a family $\mathcal C := \{ M_k \}_{ k \in \R }$. 
%We shall always assume that \textit{$K$ is non-constant.}

\begin{definition}\label{def:pKc}
We say that a surface 
$\xbf \colon M \to \R^3$  has \textit{parallel K-contours} if 
there exists a plane in $\R^3$ such that the orthogonal projection 
 $\pi \colon \R^3 \to P$ maps $\mathcal C$ to a family 
of parallel straight lines on $P$. 
\end{definition}

\subsection{ An example}

Let $k$, $c$ be non-zero real numbers. 
Consider a graph surface of 
\begin{equation*}
 z = c e^{kx} \cos ky,  
\end{equation*}  
that is, 
\begin{equation*}%\label{eq:pKc}
\pbf_{k,c}(x,y) = \left( x, y , c e^{kx} \cos ky  \right).  
\end{equation*}
See Figure~\ref{fg:line}.
\begin{figure}[htbp] %%%%%%%%%%
\centering
 \includegraphics[width=.40\linewidth]{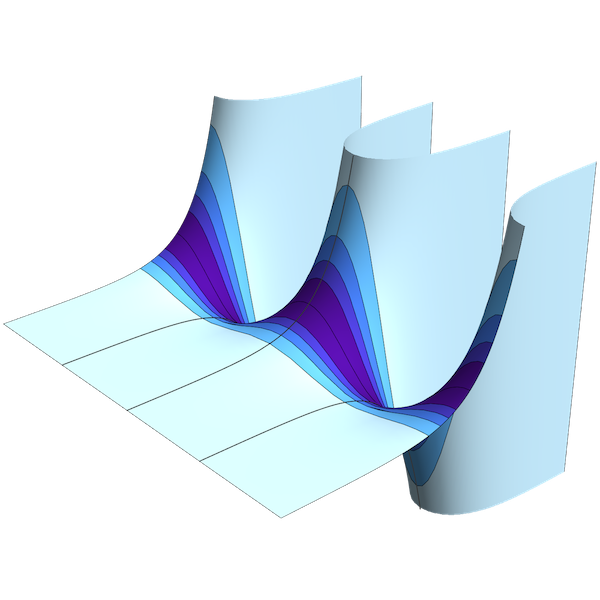} \ 
 \includegraphics[width=.40\linewidth]{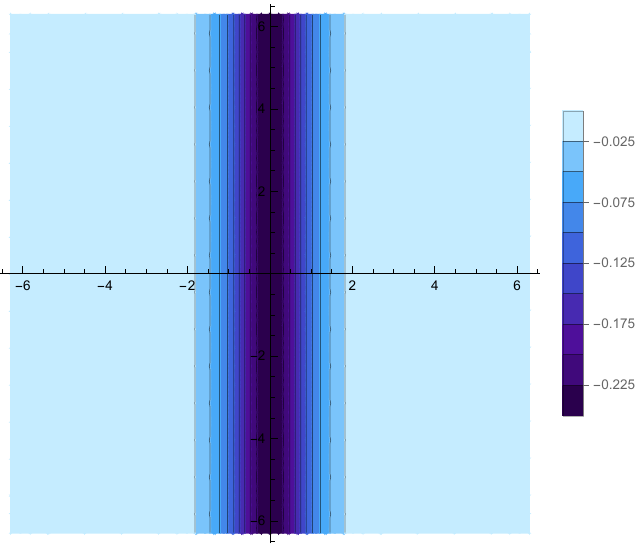} 
\caption{
The surface $\pbf_{k,c}$ (left) and its projection to the $xy$-plane (right) 
for $c=1$ and $k=1$. 
The surface is gray-scaled by its Gaussian curvature.}
\label{fg:line} %%%%%%%%%%
\end{figure} 
The first and second fundamental forms $\fff, \sff$ and a unit normal $\nbf$ 
are as follows:  
\begin{align*}
& \begin{aligned}
 \fff = \ & (1+ c^2 k^2 e^{2kx} \cos^2 ky) dx^2 
-2 c^2 k^2 e^{2kx} \cos ky \sin ky \, dxdy \\
& + (1+ c^2 k^2 e^{2kx} \sin^2 ky) dy^2 , 
\end{aligned}
\\
 & \nbf(x,y) = \frac{1}{\sqrt{ 1+ c^2 k^2 e^{2kx} }}
\left( -c k e^{kx} \cos ky, c k e^{kx} \sin ky, 1 \right) , %\label{al:Gmap-pk} 
\\
&  \sff = \frac{c k^2 e^{kx}}{\sqrt{ 1+ c^2 k^2 e^{2kx} }} 
\left( \cos ky \, dx^2 -2 \sin ky \, dxdy - \cos ky \, dy^2  \right). 
\end{align*}
From these, the Gaussian curvature $K$ and the mean curvature $H$ are 
\begin{align}
& K = K(x) = - \frac{c^2 k^4 e^{2kx}}{(1+c^2 k^2 e^{2kx})^2}, \label{al:Gauss-pKc} 
\\ 
& H = H(x, y) = - \frac{c^3 k^4 e^{3kx} \cos ky}{2(1+c^2 k^2 e^{2kx})^{3/2}} . 
\notag 
\end{align}
It follows directly from \eqref{al:Gauss-pKc} that 
$\pbf_{k,c}$ has parallel $K$-contours.    

\subsection{Characterizations of the surface $\pbf_{k,c}$}

An assertion similar to Theorem~\ref{thm:cKc-dAinv}  
holds for surfaces with parallel $K$-contours: 
\begin{theorem}\label{thm:pKc-dAinv}
 Let $\xbf \colon M \to \R^3$ be a 
surface with parallel $K$-contours. 
If the area element $dA$ is invariant along each $K$-contour, 
then $\xbf$ is locally congruent to a surface $\pbf_{k,c}$
for some $k,c$. 
\end{theorem}

We omit the proof because it is quite similar to that of 
Theorem \ref{thm:cKc-dAinv} by 
discussing about a graph surface $(x,y,F(x,y))$. 

\bigskip

As well as a surface $\xbf_{m,c}$ in Section 2, 
the Gauss map $\nbf$ of a surface $\pbf_{k,c}$ 
satisfies the following property: 
\begin{equation*}
 \nbf (x, y + \alpha) = \mathcal R_{-k \alpha} \circ \nbf (x , y), 
\end{equation*}
where $\mathcal R_{-k \alpha}$ denotes the rotation of angle $-k \alpha$ 
with respect to the $z$-axis. 
Focusing on this property, we give the following definition:
\begin{definition}
 An immersed surface $\xbf \colon M \to \R^3$ is said to have 
\textit{quasi-rotational Gauss map} if there exist 
a straight line $l \subset \R^3$, 
a plane $P \subset \R^3$, a vector $\vbf$ parallel to $P$, and a $1$-parameter group 
$\{ \phi_t \}$ of diffeomorphisms of $M$ such that 
\begin{enumerate}
 \item $l$ is orthogonal to $P$, 
 \item $\pi \circ \xbf \circ \phi_t = T_{t \vbf} \circ \pi \circ \xbf$, and
 \item $\nbf \circ \phi_t = \mathcal R_{k t} \circ \nbf$ for some constant $k$
\end{enumerate}
with a suitable choice of orientations of $l$ and $P$,  
where $\pi \colon \R^3 \to P$ is the orthogonal projection,  
$T_{t \vbf}$ denotes a parallel translation on $P$ 
of the translation vector $t \vbf$, 
and $\mathcal R_{k t}$ denotes a rotation in $\R^3$ of angle $kt$  
with respect to the axis $l$. 
\end{definition}
 
Note that a cylindrical surface has quasi-rotational Gauss map with $k=0$, 
however it should be said to have \textit{parallel} Gauss map. 
So we shall use the term \textit{`strictly quasi-rotational'} 
in the sense of `quasi-rotational but not parallel'. 

An assertion similar to Corollary \ref{cor:semi-rot-G} holds for surfaces with 
strictly quasi-rotational Gauss map. 

\begin{theorem}\label{thm:quasi-rot-G}
 Let $\xbf \colon M \to \R^3$ be a surface   
with strictly quasi-rotational Gauss map.  
Then $\xbf$ is locally congruent to a surface $\pbf_{k,c}$
for some $k,c$.  
\end{theorem}

We omit the proof because 
it is quite similar to that of 
Theorem \ref{thm:semi-rot-G} by 
discussing about a graph surface $(x,y,F(x,y))$.

\bigskip
\bigskip
\bigskip

\textbf{Acknowledgments:}
The authors would like to thank Professor Wayne Rossman for his helpful comments.

The first author was 
supported by Grant-in-Aid for 
Scientific Research (C) No.~21K03226
 from the Japan Society for the Promotion of Science.
The second author was 
supported by Grant-in-Aid for 
Scientific Research (C) No.~23K03086
 from the Japan Society for the Promotion of Science.
The third author was 
supported by Grant-in-Aid for 
Scientific Research (C) No.~20K03617
 from the Japan Society for the Promotion of Science.

%
%\bigskip
%\bigskip
%
%
%\textbf{Author Contribution:}  
%%This article is due to the equal contribution of the authors.
%All authors contributed to the study conception and design.
%
%\bigskip
%\bigskip
%
%
%% ann 1
%\textbf{Conflict of Interest Statement:} 
%On behalf of all authors, the corresponding author states that 
%there is no conflict of interest. 
%
% 
%\bigskip
%\bigskip
%
%\textbf{Data Availability Statement:} 
%Data sharing not applicable to this article as no datasets 
%were generated or analysed during the current study. 
%
%\bigskip
%\bigskip
%
%
%
%\begin{comment}
%
%% ann 2
%Conflict of Interest Statement: 
%
%\end{comment}
%

\bibliographystyle{amsplain}

\end{document}